\pdfoutput=1
\documentclass[11pt]{article}
\usepackage{amsmath}
\usepackage{amssymb}
\usepackage{amsthm}
\usepackage{enumerate}
\usepackage[pdftex]{graphicx,color}
\usepackage{wrapfig}

\newtheorem{theorem}{Theorem}[section]
\newtheorem{proposition}[theorem]{Proposition}
\newtheorem{lemma}[theorem]{Lemma}
\newtheorem{corollary}[theorem]{Corollary}

\newcommand{\ind}[1]{\mathbf{1}_{\{#1\}}}
\newcommand{\pr}[1]{\operatorname{\mathbf{Pr}}\left[#1\right]}

\newcommand{\prcond}[2]{\operatorname{\mathbf{Pr}}\left[#1\;\middle\vert\;#2\right]}

\newcommand{\E}[1]{\operatorname{\mathbf{E}}\left[#1\right]}
\newcommand{\Espace}[2]{\operatorname{\mathbf{E}}_{#2}\left[#1\right]}

\newcommand{\vol}[1]{|#1|}

\newcommand{\critperc}{\mathrm{c}}
\newcommand{\compl}{\textrm{c}}
\newcommand{\tcov}{T_\mathrm{cov}}
\newcommand{\tdet}{T_\mathrm{det}}
\newcommand{\tperc}{T_\mathrm{perc}}
\newcommand{\tbroad}{T_\mathrm{bc}}

\def\rset{{\mathbb{R}}}
\def\zset{{\mathbb{Z}}}

\newcommand{\ignore}[1]{}
\newcommand{\topic}[1]{\par\medskip\noindent{\bf #1}\par\smallskip\noindent}

\begin{document}

\title{Mobile Geometric Graphs, and Detection and Communication Problems in Mobile Wireless Networks}

\author{Alistair Sinclair\thanks{Computer Science Division, University
        of California, Berkeley CA~94720-1776, U.S.A.\ \ Email:
        \hbox{sinclair@cs.berkeley.edu}.
        Supported in part by NSF grant CCF-0635153
        and by a UC Berkeley Chancellor's Professorship.}
\and
        Alexandre Stauffer\thanks{Computer Science Division, University
        of California, Berkeley CA~94720-1776, U.S.A.\ \ Email:
        \hbox{stauffer@cs.berkeley.edu}.
        Supported by a Fulbright/CAPES scholarship and NSF grants CCF-0635153 and 
        DMS-0528488.}
}
%\date{}
\maketitle
\thispagestyle{empty}

\begin{abstract}
Static wireless networks are by now quite well understood mathematically
through the {\it random geometric graph\/} model.  By contrast,
there are relatively few rigorous results on the practically important 
case of {\it mobile\/} networks, in which 
the nodes move over time; moreover, these results often make unrealistic
assumptions about node mobility such as the ability to make very large 
jumps.  In this paper we consider a realistic model for mobile wireless 
networks which we call {\it mobile geometric graphs}, and which is a 
natural extension of the random geometric graph model.  
We study two fundamental questions in this model: {\it detection\/} 
(the time until a given ``target" point---which may be either 
fixed or moving---is detected by the network), and {\it percolation\/} 
(the time until a given node is able to communicate with the giant component
of the network).  
For detection, we show that the probability that the detection time
exceeds~$t$ is $\exp(-\Theta(t/\log t))$ in two dimensions, and 
$\exp(-\Theta(t))$ in three or more dimensions, under reasonable
assumptions about the motion of the target.  For percolation, we show
that the probability that the percolation time exceeds~$t$ is
$\exp(-\Omega(t^\frac{d}{d+2}))$ in all dimensions $d\geq 2$.  We also give a
sample application of this result by showing that the time required to 
broadcast a message through a mobile network with $n$ nodes above the
threshold density for existence of a giant component
is $O(\log^{1+2/d} n)$ with high probability.
\end{abstract}

\newpage
\setcounter{footnote}{0}
\setcounter{page}{1}
%############################################################################################
%############################################################################################
%############################################################################################
\section{Introduction}
A principal focus in wireless network research today is on mobile ad hoc networks, 
in which nodes moving in space 
cooperate to relay packets on behalf of other nodes without any centralized
infrastructure.  Although the {\it static\/} properties of such networks are by now
quite well understood mathematically, the additional challenges posed by
node mobility have so far received relatively little attention from the theory 
community.  In this paper we consider a mathematical model for mobile
wireless networks, which we call {\it mobile geometric graphs\/} and which
is a natural extension of the widely studied {\it random geometric graphs\/}
model of static networks.  We study two fundamental problems in this model:
the {\it detection \/} problem (time until a fixed or moving target is detected by
the network), and the {\it percolation \/} problem (time until a given node is
able to communicate with many other nodes).

In the random geometric graph (RGG) model~\cite{Penrose},
nodes are distributed in a region $S\subseteq\rset^d$ according
to a Poisson point process of intensity~$\lambda$ (i.e., the number of nodes
in any subregion $A\subseteq S$ is Poisson with mean~$\lambda|A|$, where $|A|$
is the volume of~$A$).  Two nodes are connected
by an edge iff their distance is at most~$r$, where the parameter~$r$ is 
the {\it transmission range\/} that specifies the distance over which nodes may send 
and receive information;  since the structure of the RGG depends only on
the product $\lambda|B_r|$ (where $B_r$ is the radius-$r$
ball in~$\rset^d$)~\cite{BR}, we may fix~$r$ so that $|B_r|=1$ and parameterize the
model on~$\lambda$ only.  
%For concreteness we shall focus unless otherwise stated
%on two and three dimensions ($d=2,3$, the primary cases of interest), though most 
%results extend to any dimension $d\ge 2$.  
We shall take $S$ to be a cube
of volume~$n/\lambda$ (so that the expected number of nodes in~$S$ is~$n$), 
and consider the limiting behavior as $n\to\infty$.

Clearly, increasing $\lambda$ increases the average degree of the nodes.
As is well known, there are two critical values of~$\lambda$
at which the connectivity properties of the 
RGG undergo a significant change.  First there is the {\it percolation
threshold}~$\lambda=\lambda_\critperc$ (a constant that depends on the dimension~$d$),
so that if $\lambda>\lambda_\critperc$ the network w.h.p.\footnote{We shall take the phrase
``w.h.p.'' (``with high  probability'') to mean ``with probability tending
to~1 as $n\to\infty$.''}\ has a unique ``giant'' component containing a  constant 
fraction of the nodes, while if $\lambda<\lambda_\critperc$ all components have size 
$O(\log n)$ w.h.p.~\cite{Penrose}.  Second, at the {\it connectivity threshold\/}
$\lambda=\log n$, the network becomes connected w.h.p.~\cite{GK1}.
The percolation threshold $\lambda=\lambda_\critperc$ occurs also in the infinite-volume
limit where $S=\rset^d$, in which case the giant component is the unique 
infinite component (or ``infinite cluster") with probability~1.  
These and other fundamental properties 
of RGGs are extensively discussed in the book of Penrose~\cite{Penrose};
see also~\cite{GRK} for additional results on thresholds.
There are a host of theoretical results on routing and other algorithmic
questions on static RGGs (see the Related Work section below for a
partial list).  Naturally, most of these consider networks above the
connectivity threshold.

A central feature of many ad hoc networks is the fact that the nodes are
moving in space.  This is the case, for example, in vehicular
networks (where sensors are attached to cars, buses or taxis), surveillance
and disaster recovery applications where mobile sensors are used to 
survey an area, and pocket-switched networks based on mobile communication
devices such as cellphones.  
Such networks are also frequently modeled using RGGs,  augmented by motion 
of the nodes.  We will employ the following model, which we refer to as 
{\it mobile geometric graphs\/} (MGGs) and which is essentially equivalent 
to the ``dynamic boolean model" introduced in~\cite{vandenberg}
in the context of dynamic continuum percolation.   
We begin at time~0 with a 
(infinite)\footnote{Passing to infinite volume is a standard
device that eliminates boundary effects in a finite region; with a little more technical
effort the model and results can be extended to finite regions with a suitable
convention---such as reflection or wraparound---to handle motion of nodes at 
the boundaries.  See, e.g., Corollary~\ref{cor:broadcast} below.}  
RGG~$G_0$ in $S=\rset^d$. Nodes move independently in continuous time 
according to Brownian motion with variance~$s^2$; here $s$ is a {\it range of 
motion\/} parameter, which we assume is constant to ensure a realistic model.
We observe the nodes at discrete time steps
(so the displacement of a node in each direction in each
time step is normally distributed with mean~0 and variance~$s^2$).
It is not hard to verify that this produces a Markovian sequence of RGGs 
$G_0,G_1,\ldots$, all with the same value of~$\lambda$.
Note that, while each $G_i$ is a RGG, there are correlations over time; 
it is this feature that makes mobility challenging to analyze.  

Once mobility is injected, the questions of interest naturally change from those
in the static case.  For example, connectivity no longer plays such a central
role because mobility may allow nodes $u,v$ to  exchange messages even in
the absence of a path between them at any given time: namely, $u$ can route
its message to~$v$ along a time-dependent path, opportunistically using
other nodes to relay the message towards~$v$.  Networks of this kind are often
termed ``delay tolerant networks''~\cite{Fall}.  This allows us to focus not on the
rather artificial connectivity regime (where $\lambda$ grows with~$n$), but instead
on the case where~$\lambda$ (and hence the average degree) is constant.
%the percolation regime in which $r>r_0$ is constant but there is a giant component.
%(Below the percolation threshold the model becomes less interesting in many
%applications because communication would involve the physical
%transport of packets by the moving nodes.)  
Keeping $\lambda$ constant is obviously highly 
desirable as it makes the model much more realistic and scalable.

There are rather few rigorous results on wireless networks
with mobile nodes, and those that do exist typically either make unrealistic
assumptions about node mobility (such as unbounded range of
motion~\cite{GT,DGT,Clem} or no change in direction~\cite{Liu}), 
or work in the connectivity regime which, as we have seen, requires
unbounded density or transmission range~\cite{Diaz,EMPS}.  
(See the Related Work section for more details.)  
In this paper we study two fundamental questions
for mobile networks assuming only constant average degree and bounded
range of motion (i.e., constant values of the parameters $\lambda$ and~$s$).
%%%%%%%%%%%%%%
\subsubsection*{Results}
{\bf Detection.}\ \ 
A central issue in surveillance and remote sensing applications
is the ability of the network to detect a ``target''
$u\in S$ (which may be either fixed or moving), in the
sense that there is a node within distance at most~$r$ of~$u$.  
It is well known~\cite{MR} that, for a static RGG, a fixed target can
be detected only with constant probability 
unless the average 
degree grows with~$n$.
In the mobile case, we may hope to achieve detection {\it over time\/} with 
constant average degree ($\lambda=O(1)$, even {\it below\/} the percolation threshold).
In this scenario, the detection time, $\tdet$, is formulated as the number of
steps until a target initially at the origin is detected by the MGG.
Recent work of Liu et al.~\cite{Liu} shows that the detection time
in two dimensions is exponentially distributed when the nodes 
of the network move in {\it fixed\/} directions.  
In the more realistic MGG model, we are able to prove the following
result which holds in all dimensions (see Section~\ref{sec:detection}):
\begin{theorem}\label{thm:detection}
In the MGG model with any fixed~$\lambda$ and 
range of motion~$s>0$, the detection time for a fixed target or a target
moving under Brownian motion satisfies
$\pr{\tdet \geq t} = \exp\left(-\Theta(t/\log t)\right)$ for $d=2$, and 
$\pr{\tdet \geq t} = \exp\left(-\Theta(t)\right)$ for $d \geq 3$.
\end{theorem}
The constants in the $\Theta$ here depend only on $\lambda$, $s$ 
and the dimension~$d$.  Thus the tail of the detection time is exponential
in three and higher dimensions, and exponential with a logarithmic
correction in two dimensions.  We note that, as is evident from the proof,
this dichotomy between two and three dimensions reflects the difference 
between recurrent and transient random walks in~$\zset^2$ and~$\zset^3$ 
respectively.  We also note that the upper bound in Theorem~\ref{thm:detection}
holds for {\it arbitrary\/} motion of the target (provided it is independent of
the motion of the nodes); and the lower bound holds for any ``sufficiently random"
motion of the target.
%Finally, if the motion of the
%nodes is given a positive drift in a fixed direction, then the $\log t$ correction
%in two dimensions disappears, in line with the result of~\cite{Liu}.

We should point out that, for the special case of a fixed target, a 
slightly stronger version of Theorem~\ref{thm:detection}, with a tight
constant in the exponent, follows from classical results on the ``Wiener
sausage'' in continuum percolation. 
(This was pointed out to us by Yuval Peres~\cite{Peres}; see Related
Work for details.)
However, it is not clear how to extend this approach to the case of
a moving target.  Our proof is elementary and based on an
application of the mass transport principle.
\par\medskip\noindent
{\bf Percolation.}\ \ 
A fundamental question in mobile networks is whether a node can efficiently
communicate with other nodes even when the network is not connected at any given time.
In the MGG model, this question may naturally be formulated by considering
a constant intensity $\lambda>\lambda_\critperc$ (i.e., above the percolation threshold) 
and asking how long it takes until a node initially at the origin belongs to the giant 
component (or the infinite component in the limit $n\to\infty$).  We call this 
the {\it percolation time}~$\tperc$.  It should be clear that the percolation time can be used 
to derive bounds on other natural quantities, such as the time for a node to
broadcast information to all other nodes (see Corollary~\ref{cor:broadcast}
below).  As far as we are aware the 
percolation time has not been investigated before, largely because previous work on
RGGs has focused on networks above the connectivity threshold.  However,
it appears to be a fundamental question in the mobile context.

The detection time clearly provides a lower bound on the percolation time,
so we may deduce from Theorem~\ref{thm:detection} above that 
$\pr{\tperc\geq t}$ is at least $\exp\left(-O(t/\log t)\right)$ for $d=2$ 
and at least $\exp\left(-O(t)\right)$ for $d \geq 3$.  We are able to
prove the following stretched exponential upper bound in all dimensions
$d\geq 2$ (see Section~\ref{sec:percolation}):
\begin{theorem}\label{thm:percolation}
In the MGG model with any fixed $\lambda>\lambda_\critperc$ 
and range of motion $s>0$, the percolation time for a node at the origin satisfies
$\pr{\tperc \geq t} =\exp\left(-\Omega(t^{d/(d+2)})\right)$ in all
dimensions $d\ge 2$.
\end{theorem}
Again, the constant in the  $\Omega$ depends only on $\lambda$, $s$
and~$d$.  There is a gap between this upper bound and the lower bound 
from Theorem~\ref{thm:detection}.  We conjecture that the true tail behavior
of~$\tperc$ is $\exp\left(-\Theta(t/\log t)\right)$ for $d=2$ and  
$\exp\left(-\Theta(t)\right)$ for $d \geq 3$.

Theorem~\ref{thm:percolation} is the main technical contribution of the paper;
we briefly mention some of the ideas used in the proof.  The key technical
challenge is the dependency of the RGGs~$G_i$ over time.  To overcome
this, we partition~$\rset^d$ into subregions of suitable size and {\it couple\/}
the evolution of the nodes in each subregion with those of a fresh Poisson
point process of slightly smaller intensity $\lambda'<\lambda$ which is still
larger than the critical value~$\lambda_\critperc$.  After 
a number of steps~$\Delta$ that depends on the size  of the subregion,
we are able to arrange that the coupled processes match up almost
completely.  As a result, we can conclude that our original MGG process,
observed every~$\Delta$ steps, contains a sequence of {\it independent\/}
Poisson point processes with intensity~$\lambda''>\lambda_\critperc$.  
(This fact, which we believe is of wider applicability, is formally stated
in Proposition~\ref{pro:coupling} in Section~\ref{sec:percolation}.)  This
independence is sufficient to complete the proof.  The slack in the bound
comes from the ``delay"~$\Delta$.

To illustrate a sample application of Theorem~\ref{thm:percolation}, we consider
the time taken to broadcast a message in a network of finite size.
Consider a MGG in a cube of volume~$n/\lambda$ (so the expected\footnote{The
result can be adapted to the case of a {\it fixed\/} number of nodes~$n$ using 
standard ``de-Poissonization" arguments~\cite{Penrose}.  See the Remark 
following the proof of Corollary~\ref{cor:broadcast} in Section~\ref{sec:broadcast}.}
number of nodes is~$n$).  Since the volume is finite, we need to modify
the motion of the nodes to take account of boundary effects: following
standard practice, we do this by turning the cube into a torus (so that nodes 
``wrap around" when they reach the boundaries).
Suppose a message originates at an arbitrary node
at time~0, and at each time step~$t$ each node that has already
received the message broadcasts it to all nodes in the same connected
component.  (Here we are making the reasonable assumption
that the speed of transmission is much faster than the motion of the nodes,
so that messages can travel throughout a connected component before it is
altered by the motion.)  Let $\tbroad$ denote the time until all nodes have
received the message.
\begin{corollary}\label{cor:broadcast}
In a MGG on the torus of volume~$n/\lambda$ with any fixed $\lambda > \lambda_\critperc$ and range of motion $s>0$, 
the broadcast time $\tbroad$ 
is $O(\log^{1+2/d} n)$ w.h.p.\ in any dimension $d\ge 2$.
\end{corollary}
%%%%%%%%%%
\subsubsection*{Related work}
There are many theoretical results on routing and other
algorithmic questions on (static) RGGs; we mention
just a few highlights here.  The seminal work of Gupta and Kumar~\cite{GK2,GK3}
(see also~\cite{Fran} for refinements) examined the information-theoretic
{\it capacity\/} (or throughput) of such networks above the connectivity 
threshold, i.e., the number of bits per unit time that each node~$u$ can
transmit to some (randomly chosen) destination node~$t_u$ in steady
state, assuming constant size buffers in the network.  The capacity per
unit node is $\Theta(n^{-1/2})$, which tends to~0 as $n\to\infty$, 
suggesting a fundamental limitation on the scalability of such static
networks.  

The {\it detection\/} problem has received
much attention.  In the static case detection is essentially equivalent to 
{\it coverage\/} of the region~$S$, which requires that the network be connected.
In the absence of coverage, Balister et al.~\cite{Bal} determine the  maximum 
diameter of the uncovered regions, while Dousse et al.~\cite{Dousse} prove that, for
any $\lambda>0$, the detection time for a target moving in a fixed
direction has an exponential tail.
%a target moving in an arbitrary fixed direction is detected 
%within~$t$ steps with probability $1-\exp(\Omega(t))$.  
(Note that this is not a mobility result as the nodes are fixed.)  

The question of {\it broadcasting\/} within the giant component of a RGG 
above the percolation
threshold was recently analyzed by Bradonji\'c et al.~\cite{B++}, who also show
that the graph distance between any two (sufficiently
distant) nodes is at most a constant factor larger than their Euclidean 
distance.  Cover times for random walks on (connected) RGGs were investigated by 
Avin and Ercal~\cite{AE} and Cooper and Frieze~\cite{CF2}, while
the effect of physical obstacles that obstruct 
transmission was studied by Frieze et al.~\cite{FKRD}.

The scope of mathematically rigorous work on RGGs with mobility
is much more limited.  We briefly summarize it here.

Motivated by the fact mentioned above~\cite{GK2} that the capacity of static
networks goes to zero as $n\to\infty$, Grossglauser and Tse~\cite{GT} (see
also~\cite{DGT}) showed how to exploit mobility to achieve constant 
capacity using a two-hop routing scheme.  However, these results require
the unrealistic assumption that  nodes move a distance comparable to the
diameter of the entire region~$S$ at each step.
El~Gamal et al.~\cite{EMPS} study the tradeoff between capacity and
delay in a realistic mobility model but above the connectivity threshold.
Clementi et al.~\cite{Clem} show how to exploit mobility to enable
broadcast in a RGG sufficiently far above the percolation threshold.
However, this result again assumes that the
range of motion of the nodes is unbounded (i.e., $s$ grows with~$n$).

As mentioned earlier, the detection problem was addressed
by Liu et al.~\cite{Liu}, assuming that each node moves continuously
in a {\it fixed\/} randomly chosen direction; they show that the time it takes for
the network to detect a target is exponentially distributed with expectation
depending on the intensity~$\lambda$.  Also, for the special case of
a stationary target, as observed in~\cite{Kesidis,Konst} a slightly 
stronger version of Theorem~\ref{thm:detection}, with tight constants
in the exponent, can be deduced from 
classical results on continuum percolation: namely, in this case it is shown
in~\cite{SKM} that (in continuous time)
$\mathbf{Pr}[\tdet \geq t] = \exp(-\lambda V_s(t))$, where $V_s(t)$ is the 
expected volume of the ``Wiener sausage" of length~$s^2t$ (essentially the
trajectory of a Brownian motion ``fattened" by a disk of radius~$r$).  This 
volume in turn is known quite precisely~\cite{Spitzer,BMS}.

A model essentially equivalent to MGGs was introduced under the name
``dynamic boolean model" by Van den Berg et al.~\cite{vandenberg},
who studied the measure of the set of times at which an infinite component
exists.
Finally, recent work of D\'iaz et al.~\cite{Diaz} in a similar model determines, for 
networks exactly at the connectivity threshold, the expected length of time for which 
the network  stays connected (or disconnected) as the nodes move.
However, this question makes sense only for very large values of~$\lambda$
(growing with~$n$) and thus falls outside the scope of our investigations.
%as a function of the range and mobility parameters~$r$ and~$s$.
%############################################################################################
%############################################################################################
%############################################################################################
\section{Preliminaries}
\label{sec:prelims}
%############################################################################################
%\subsubsection*{Basic geometric notation}
%\vskip-0.05in
For any $\ell\geq 0$, let $B_\ell$ be the $d$-dimensional ball centered at the origin with radius~$\ell$. 
Similarly, let $Q_\ell$ be the cube with side-length $\ell$ centered at the origin and with sides
parallel to the axes of $\mathbb{R}^d$. For any point
$z\in \mathbb{R}^d$ and set $A \subseteq \mathbb{R}^d$, 
we define $z+A$ as the Minkowski sum $z+A = \{y \colon y-z\in A\}$.
%In other words, we can view $z$ as a vector and obtain $z+A$ by taking the set $A$ 
%and shifting it according to the vector $z$. For example, $z+B_\ell$ is the ball of 
%radius $\ell$ centered at $z$. 
The volume of a set~$A\subset\mathbb{R}^d$ is denoted~$|A|$.
%We consider $r=r(d)$ to be a number such that $\vol{B_r}=1$.

%############################################################################################
%\subsubsection*{Poisson point processes}
%\vskip-0.05in
\topic{Poisson point processes}%
A ``point process" is a random collection of points in~$\mathbb{R}^d$; for a formal
treatment of this topic, the reader is referred to~\cite{SKM}. To avoid ambiguity, we 
refer to the points of a point process as \emph{nodes} and reserve the word \emph{points} for
arbitrary locations in $\mathbb{R}^d$. 
We are mostly interested in \emph{Poisson point processes}. A Poisson
point process with intensity~$\lambda$ in a region $S\subseteq\rset^d$
is defined by a single property: 
for every bounded Borel set $A\subseteq S$, the 
number of points in~$A$ is a Poisson random variable with mean~$\lambda \vol{A}$. 
We will make use of the following standard properties of Poisson point processes:
(1)~for disjoint sets $A, A'\subseteq S$, the numbers of points in~$A$ and in~$A'$ are independent;
(2)~conditioned on the number of nodes in~$A$, each such node is located independently
and uniformly at random in~$A$;
(3)~[thinning] if each node of a Poisson point process with intensity~$\lambda$ is deleted
with probability~$p$, the result is a Poisson point process with intensity~$(1-p)\lambda$;
(4)~[superposition] the union of two Poisson point processes in~$S$ with intensities 
$\lambda_1$ and $\lambda_2$ is a Poisson point process with intensity $\lambda_1+\lambda_2$.
In some of our proofs we will make use of {\it non-homogeneous\/} Poisson point processes,
whose intensity $\lambda(x)$ is a function of position $x\in\rset^d$.  In such a process
the expected number of nodes in a set~$A$ is $\int_A\lambda(x)dx$.
%############################################################################################
%\subsubsection*{Random geometric graphs}
%\vskip-0.05in
\topic{Random geometric graphs}%
Fix parameters $\lambda,r\geq 0$, and let
$S_n=Q_{(n/\lambda)^{1/d}}$ be the cube of volume~$n/\lambda$ in~$\rset^d$.
Let $\Xi_n$ be a Poisson point process over~$S_n$ with 
intensity~$\lambda$. A random geometric graph (RGG)
$\mathcal{G}(\Xi_n,r)$ is constructed by taking the node set to be the nodes of $\Xi_n$ and 
creating an edge between every pair of nodes whose Euclidean distance is at most~$r$.
The parameter $r$ is called the \emph{transmission range}.
Since $\Xi_n$ is a Poisson point process, 
the expected number of nodes in $\mathcal{G}(\Xi_n,r)$ is~$n$. 

It is well known~\cite{BR,MR} that as $n\to\infty$
the random graph model induced by $\mathcal{G}(\Xi_n,r)$ depends only 
on the product $\lambda \vol{B_r}$.  For this reason, we will always fix $r=r(d)$
so that $\vol{B_r}=1$ and parameterize the model only on~$\lambda$. 
Note that with this convention, in the limit as $n\to\infty$, $\lambda$ is also
the expected degree of any node in $\mathcal{G}(\Xi_n,r)$.

Using a Poisson point process rather than a fixed number of nodes is a standard
trick that simplifies the mathematics.  Most results in this model can be
translated to a model with a fixed number~$n$ of nodes in~$S_n$
using a technique known as ``de-Poissonization"~\cite{Penrose}.

Many asymptotic properties of 
$\mathcal{G}(\Xi_n,r)$ as $n\to\infty$ are studied in the monograph by Penrose~\cite{Penrose}. 
For example, it is known that $\lambda=\log n$ is a threshold for connectivity, in the sense
that if $\lambda = \log n+\omega(1)$ then $\mathcal{G}(\Xi_n,r)$ is connected 
w.h.p., and if $\lambda = \log n - \omega(1)$ then $\mathcal{G}(\Xi_n,r)$
is disconnected w.h.p.  Another important critical value is the 
{\it percolation threshold\/} $\lambda=\lambda_\critperc$ (a constant that depends on
the dimension~$d$); if $\lambda > \lambda_\critperc$ then w.h.p.\ $\mathcal{G}(\Xi_n,r)$ 
contains a unique ``giant" connected component with $\Theta(n)$ nodes, 
while all other components are of size $O(\log^{d/(d-1)}n)$;
on the other hand, if $\lambda < \lambda_\critperc$ then w.h.p.\ 
all connected components of $\mathcal{G}(\Xi_n,r)$ have size $O(\log n)$.
The value of $\lambda_\critperc$ is not known exactly in any dimension $d\ge 2$.
However, for $d=2$ the rigorous bounds $2.18 < \lambda_\critperc < 10.60$ are
known~\cite[Section~3.9]{MR}, while Balister et al.~\cite{BalBol2} used Monte 
Carlo methods to deduce that $\lambda \in (4.508,4.515)$ with confidence $99.99\%$.

Finally, we remark that in the limit as $n\to\infty$ 
(that is, when the Poisson point process is defined over the whole of $\mathbb{R}^d$) 
the percolation threshold $\lambda=\lambda_\critperc$ still exists and is characterized by
the appearance of a unique infinite component (or ``infinite cluster") with probability~1
for any $\lambda>\lambda_\critperc$.  In this limit the graph is disconnected with probability~1
for any value of~$\lambda$.

%############################################################################################
%\subsubsection*{Mobile geometric graphs}
%\vskip-0.05in
\topic{Mobile geometric graphs}%
We define our mobile geometric graph (MGG) model by taking a Poisson point process 
with intensity~$\lambda$ in $\mathbb{R}^d$ at time~0 and letting each node move in continuous 
time according to an independent Brownian motion.  We sample the locations of
the nodes at discrete time steps $i=1,2,\ldots$ and use these locations to define a
sequence of random geometric graphs with transmission range~$r=r(d)$.  We base our
MGG model on the infinite volume~$\rset^d$ to avoid having to handle boundary
effects on the motion of the nodes.  Results in this model can be translated to finite
regions with a suitable convention---such as wraparound or reflection---to handle the
motion of nodes at the boundaries.  

More formally, let $\Pi_0$ be a Poisson point process with intensity $\lambda$ 
over~$\mathbb{R}^d$.  We take a parameter $s\geq 0$, and with each node $v\in \Pi_0$
we associate an independent $d$-dimensional Brownian motion $\{W_v(i)\}_{i\geq 0}$ 
that starts at the location of $v$ in $\Pi_0$ and has variance~$s^2$~\cite{KS}. 
%This means that for any time $t$, $W_v(t) = (W_v^{(1)}(t), W_v^{(2)}(t), \ldots, W_v^{(d)}(t))$, 
%and each $W_v^{(j)}(t)$, $1\leq j\leq d$, is 
%a $1$-dimensional Brownian motion with mean $0$ and variance $s^2$.
Now, for any $i\in \mathbb{Z}^+$, we define $\Pi_i$ as the point process obtained by 
putting a node at $W_v(i)$ for each $v\in\Pi_0$.  A MGG is then the collection of
graphs $G=\{G_i\}_{i\geq 0}$ where $G_i=\mathcal{G}(\Pi_i,r)$ and $r=r(d)$ is
fixed so that $|B_r|=1$. 
(Note that, as in the static case, fixing the value of~$r$ may be done w.l.o.g.)
%(Note that, as in the static case, fixing the value of~$r$ may be done w.l.o.g.\
%as it can be shown that the MGG model depends only on the 
%\marginpar{\tiny We should discuss this point, per your email.  AS}
%product~$\lambda\vol{B_r}$ and the ratio~$r/s$.)

It is an easy consequence of the mass transport principle (see below) that 
each $\Pi_i$, viewed in isolation, is itself a Poisson point process with intensity~$\lambda$.
This means that the sequence $\{\Pi_i\}_{i\geq 0}$ is stationary and therefore, 
when viewed in isolation, $G_i$ is a random geometric graph over $\mathbb{R}^d$.
Thus, for example, if $\lambda > \lambda_\critperc$ then each $G_i$ contains an infinite component 
with probability~1.

%############################################################################################
%\subsubsection*{Mass transport principle}
%\vskip-0.05in
\topic{Mass transport principle}%
For two points $x,y\in\mathbb{R}^d$ and a time step $i \geq 0$, we define $f_i(x,y)$ as the probability
density function for a node located at position $x$ at time $0$ to be at position $y$ at time $i$. 
Since nodes move according to $d$-dimensional Brownian motion, we have
$f_i(x,y) = \frac{1}{(2\pi s^2 i)^{d/2}}\exp(-\frac{\|y-x\|_2^2}{2s^2i})$.
%$$
%   f_i(x,y) = \frac{1}{(2\pi s^2 i)^{d/2}}\exp\left(-\frac{\|y-x\|_2^2}{2s^2i}\right).
%$$

In some situations, it is useful to regard $f_i(x,y)$ as a mass transport function. 
For example, suppose nodes are initially distributed according to a Poisson point
process with intensity~$\lambda$ in a region $A\subseteq\rset^d$; we may view
this as a Poisson point process over~$\rset^d$ with (non-homogeneous) intensity
(or ``mass function")
$\nu_0(x)=\lambda\ind{x\in A}$.  Using the thinning  and superposition
properties, it is easy to check that the distribution of
the nodes at time~$i$ is a Poisson point process with intensity
$\nu_i(y) = \int_{\rset^d}\nu_0(x)f_i(x,y)dx$.
This interpretation can be used, for example, to show that in a MGG $\Pi_i$ is 
a Poisson point process with intensity~$\lambda$ for all~$i$, as claimed above. 

%############################################################################################
%############################################################################################
%############################################################################################
\section{Detection time}
\label{sec:detection}
In this section we prove Theorem~\ref{thm:detection}.
We consider the detection time for a node $u$ initially placed at the origin 
independently of the MGG~$G$. We say that a node $v \in G$ {\it detects}~$u$
at time step $i$ if the distance between $u$ and~$v$ at time step~$i$ is at most~$r$, and
we define $\tdet$ as the first time that $u$ is detected by some node of~$G$.
Our goal is to derive tight bounds for the tail of $\pr{\tdet \geq t}$. In the proof we 
consider the cases where $u$ is either non-mobile or moves according to Brownian motion with
variance~$s^2$. We discuss some extensions at the end of the section.

It will be convenient to restrict attention to the nodes of~$G$ that are initially inside 
the cube~$Q_L$, where $L$ is a suitably chosen parameter. We define
$\tdet(Q_L)$ as the first time a node initially inside $Q_L$ detects $u$. Note that clearly
$\pr{\tdet(Q_L) \geq t} \geq \pr{\tdet \geq t} = \lim_{L\to\infty} \pr{\tdet(Q_L) \geq t}$, where the limit exists 
since $\pr{\tdet(Q_L) \geq t}$ is monotone and bounded as a function of $L$. 
We let $X=(x_0,x_1,\ldots,x_{t-1})$ be the locations of $u$ in the first $t$ steps.
The following lemma relates the tail of $\tdet(Q_L)$ to the tail of an analogous 
random variable for a single random node in~$Q_L$.
%We prove the  following lemma.

\begin{lemma}\label{lem:det1step}
   We have $\prcond{\tdet(Q_L) \geq t}{X}=\exp\left(-\lambda L^d \prcond{\tau < t}{X}\right)$
   and $\pr{\tdet(Q_L) \geq t}\geq \exp\left(-\lambda L^d \pr{\tau < t}\right)$,
   where $\tau$ is the first time that a node initially located u.a.r.\ 
   in $Q_L$ detects~$u$.
\end{lemma}
\begin{proof}%[\textbf{Proof of Lemma~\ref{lem:det1step}}]
   Let $N$ be the number of nodes inside $Q_L$ at time $0$. Each of these nodes is initially 
   located uniformly at random inside $Q_L$, and the motion of each node does not depend 
   on the locations of the other nodes. 
   If we fix a given value for $X$, then 
   the first time that a given node $v$ detects $u$ does not depend on the other nodes of $G$ 
   and is distributed according to the conditional distribution of~$\tau$ given~$X$. 
   (Note that this is not true if $X$ is not fixed but random, because
   the random motion of~$u$ makes the relative displacements of the nodes of $G$ 
   with respect to~$u$ dependent.)
   Therefore, conditioning on $N$ and $X$, we have 
   $\prcond{\tdet(Q_L) \geq t}{N=n, X} = \prcond{\tau \geq t}{X}^n$, which yields
   $$
      \prcond{\tdet(Q_L) \geq t}{X} 
      = \Espace{\prcond{\tau \geq t }{X}^N}{N} 
      = \exp\left(-\lambda L^d \prcond{\tau < t}{X}\right),
      %\label{eq:deteq}
   $$
   where we use the notation $\Espace{\cdot}{N}$ to denote expectation with respect to the 
   random variable~$N$, and the last equality holds since $N$ is Poisson with 
   mean $\lambda \vol{Q_L} = \lambda L^d$.
   For the lower bound, we appeal to Jensen's inequality to obtain
   \begin{eqnarray*}
      \pr{\tdet(Q_L) \geq t}
      &=& \Espace{\prcond{\tdet(Q_L) \geq t}{X}}{X} \nonumber\\
      &\geq& \exp\left(-\lambda L^d \Espace{\prcond{\tau < t}{X}}{X}\right)\\
      &=& \exp\left(-\lambda L^d \pr{\tau < t}\right),
      %\label{eq:detlb}
   \end{eqnarray*}
   which completes the proof.
\end{proof}

We now proceed to derive upper and lower bounds for $\pr{\tau < t}$.
Let $v$ be a node initially located u.a.r.\ in~$Q_L$.
For time steps $i_1 \leq i_2$, let $M(i_1,i_2)$ be the expected number of time 
steps from $i_1$ to $i_2$ at which $v$ detects~$u$. We bound $M(0,t-1)$ as follows.

\begin{lemma}\label{lem:detm}
   Let $t\in\mathbb{Z}^+$ and $X=(x_0,x_1,\ldots,x_{t-1})\in \mathbb{R}^{dt}$ be arbitrary.  
   There exists a constant $c=c(d)$ and $L_0=L_0(t,\max_i \|x_i\|_2)$ such that
$ c t/L^d \leq M(0,t-1) \leq t/L^d$ for all $L \geq L_0$.
%   $$
%      c t/L^d \leq M(0,t-1) \leq t/L^d \qquad\hbox{\it for all $L \geq L_0$}.
%   $$
\end{lemma}
\begin{proof}%[\textbf{Proof of Lemma~\ref{lem:detm}}]
   We use the mass transport principle. We assume the initial intensity
   $\nu_0(z) = \lambda \ind{z \in Q_L}$ and let $\nu_i(z)$ be the intensity
   at $z\in\mathbb{R}^d$ at time $i$, i.e., $\nu_i(z) = \int_{\mathbb{R}^d} \nu_0(y) f_{i}(y,z) dy$
   for $i\geq 1$.
   At any time $i$, the probability that 
   $v$ detects $u$ is given by the ratio between the amount of mass inside $x_i+B_r$ and 
   the total amount of mass $\lambda L^d$. Noting that for $i=0$ this 
   ratio is $\frac{\lambda \vol{B_r}}{\lambda L^d}=1/L^d$,
   we can write $M(0,t-1)$ as 
   \begin{eqnarray}
      M(0,t-1) 
      &=& \frac{1}{L^d} + \sum_{i=1}^{t-1} \frac{\int_{B_r} \nu_i(x_i+z) dz}{\lambda L^d}\nonumber\\
      &=& \frac{1}{L^d} + \frac{1}{\lambda L^d} \sum_{i=1}^{t-1} \int_{B_r} \int_{Q_L} \lambda f_{i}(y,x_i+z) dy dz\nonumber\\
      &=& \frac{1}{L^d} + \frac{1}{L^d} \sum_{i=1}^{t-1} \int_{B_r} \int_{x_i+z+Q_L} f_{i}(0,y') dy' dz,
      \label{eq:proofm}
   \end{eqnarray}
   where the last step follows from the translation-invariance property of~$f_i$.
   Since $\vol{B_r}=1$, we obtain the upper bound from $\int_{x_i+z+Q_L} f_{i}(0,y')dy' 
   \leq \int_{\mathbb{R}^d} f_{i}(0,y')dy'=1$.

   For the lower bound, let $A_i=B_{s\sqrt{i+1}}$. We assume that $L$ is sufficiently large such that 
   $A_i \subseteq x_i+z+Q_L$.  Then replacing the integral over $x_i+z+Q_L$ in (\ref{eq:proofm}) by 
   an integral over $A_i$ gives the lower bound
   $$
      M(0,t-1) 
      \geq \frac{1}{L^d} + \frac{1}{L^d}\sum_{i=1}^{t-1} \int_{B_r}\int_{A_i} f_{i}(0,y')dy'dz
      \geq \frac{1}{L^d} + \frac{1}{L^d}\sum_{i=1}^{t-1} \frac{\vol{A_i}\vol{B_r}}{(2\pi s^2i)^{d/2}} \exp\left(-1\right)
      \geq \frac{ct}{L^d},
      %\label{eq:mlb}
   $$
   where we used the fact that $\|y'\|_2 \leq s\sqrt{i+1}$ for all $y'\in A_i$, and the result holds for 
   some constant $c=c(d)$.
\end{proof}

Our goal is to write $M(i_1,i_2)$ conditioning on~$\tau$.
Let $M'(y,i_1,i_2)$ be the expected number of time steps from $i_1$ to~$i_2$ 
at which $v$ detects~$u$ given that the relative location of~$v$ with respect 
to~$u$ at time $i_1-1$ is~$y$. 
The next lemma gives lower and upper bounds for $M'(Y_i,i+1,i+t)$. 
\begin{lemma}\label{lem:detm1m2}
   Let $i\in\mathbb{Z}^+$ be arbitrary. There exists an integer $t_0=t_0(d,s)$ 
   such that for all $t \geq t_0$ the following holds. 
   There exist functions $m_1(t)$ and $m_2(t)$ such that 
   $m_1(t) \leq M'(Y_i, i+1, i+t) \leq m_2(t)$ uniformly over $Y_i \in B_r$. Moreover, 
   there are constants $c_1=c_1(d)$ and $c_2=c_2(d)$ such that 
   $$
      \begin{array}{cc}
         m_1(t) \geq \left\{\begin{array}{rl}
            c_1 \log t/s^d, & \textrm{for $d=2$}\\
            c_1/s^d, & \textrm{for $d \geq 3$}
            \end{array}\right.
         &
         m_2(t) \leq \left\{\begin{array}{rl}
            c_2 \log t/s^d, & \textrm{for $d=2$}\\
            c_2/s^d, & \textrm{for $d \geq 3$}
            \end{array}\right.
      \end{array}
   $$
   The bounds for
   $m_1(t)$ hold both for the case where $u$ does not move and for the case where 
   $u$ moves according to Brownian motion with variance $s^2$. 
   The bounds for $m_2(t)$ hold uniformly over $X$.
\end{lemma}
\begin{proof}%[\textbf{Proof of Lemma~\ref{lem:detm1m2}}]
  % We assume that at time $i$, $u$ is located at the origin and $v$ is located at $Y_i$.
   For any $j \in [1,t]$, let $I_j$ be the indicator random variable for the event 
   that $v$ detects $u$ at time $i+j$, assuming that at time~$i$ $u$ is located 
   at the origin and $v$ is located at $Y_i$.
   Clearly, 
   $M'(Y_i,i+1,i+t)=\sum_{j=1}^t \E{I_j}$. Recall that $x_{i+j}$ is the location of $u$ at time $i+j$. Hence,
   $$
      \E{I_j} 
      = \int_{B_r} \frac{1}{(2\pi s^2j)^{d/2}}\exp\left(-\frac{\|x_{i+j}+z-Y_i\|_2^2}{2s^2 j}\right) dz
      \leq \frac{\vol{B_r}}{(2\pi s^2j)^{d/2}}
      = \frac{1}{(2\pi s^2j)^{d/2}}.
      %%\label{eq:proofm1m2}
   $$
   The upper bound follows by setting $m_2(t) = \sum_{j=1}^t \frac{1}{(2\pi s^2j)^{d/2}}$. 
   Note that this upper bound holds for arbitrary~$X$.

   Now we derive the lower bound. We use the fact that $Y_i,z\in B_r$. If $u$ is non-mobile, then 
   $\|x_{i+j}\|_2=\|x_i\|_2 = 0$ (recall that we assume $x_i$ to be the origin) 
   and from the triangle inequality we obtain 
   $\|x_{i+j}+z-Y_i\|_2 \leq \|x_{i+j}\|+\|z\|_2+ \|Y_i\|_2 \leq 2r$. Thus,
   $\E{I_j} \geq \frac{1}{(2\pi s^2j)^{d/2}}\exp\left(-\frac{2r^2}{s^2 j}\right)$.
   We take $j_0$ to be the smallest integer such that $j_0 \geq r^2/s^2$, set 
   $m_1(t) = \sum_{j=j_0}^t \frac{1}{(2\pi s^2j)^{d/2}} \exp(-2)$, and the result follows since 
   $t$ is sufficiently large with respect to $j_0$.   

   If $u$ moves according to a Brownian motion with variance~$s^2$, 
   we average over $x_{i+j}$ to get
   $$
      \E{I_j} 
      = \int_{\mathbb{R}^d} \int_{B_r} \frac{1}{(2\pi s^2j)^{d/2}}
         \exp\left(-\frac{\|x_{i+j}+z-Y_i\|_2^2}{2s^2 j}\right) 
         \frac{1}{(2\pi s^2 j)^{d/2}} \exp\left(-\frac{\|x_{i+j}\|_2^2}{2s^2j}\right) dzdx_{i+j}.
   $$
   Let $a>0$ be a constant and let $A_j = B_{a s\sqrt{j}-2r}$. We set $a \geq 4$ so that 
   $a s\sqrt{j}-2r \geq a s\sqrt{j}/2$ for all $j\geq j_0=\lceil r^2/s^2\rceil$.
   We integrate over $A_j$ instead of $\mathbb{R}^d$ and then use the 
   simple bounds $\|x_{i+j}+z-Y_i\|_2 \leq a s \sqrt{j}$ and $\|x_{i+j}\|_2 \leq a s \sqrt{j}$
   for all $x_{i+j}\in A_j$ and $z,Y_i\in B_r$
   to obtain
   $$
      \E{I_j} 
      \geq \frac{\vol{A_j}\vol{B_r}}{(2\pi s^2j)^{d}}\exp\left(-a^2\right) 
      \geq \frac{a'}{(s^2j)^{d/2}},
   $$
   for some constant $a'=a'(d)$.
   Now, we set 
   $m_1(t) = \sum_{j=j_0}^t \frac{a'}{(s^2j)^{d/2}}$
   and the result follows since $t$ is sufficiently large 
   with respect to $j_0$. 
\end{proof}

\par\noindent
{\bf Remark:} The bounds for $M'(Y_i,i+1,i+t)$ change substantially from $d=2$ to $d\geq 3$, 
reflecting the dichotomy between recurrent and transient random walks in~$\mathbb{Z}^2$ 
and~$\mathbb{Z}^3$: $v$~returns to a neighborhood of~$u$ infinitely often for
$d=2$ and only finitely often for $d\geq 3$.  (Note that $M'$ measures the expected
number of returns of~$u$ to a neighborhood around~$v$ in a given time interval.)
\par\smallskip
We now use Lemma~\ref{lem:detm1m2} to derive upper and lower bounds for 
$\pr{\tau < t}$.
\begin{lemma}\label{lem:dettau}
   Let the functions $m_1$ and $m_2$ be as in Lemma~\ref{lem:detm1m2}.
   For any constant $\alpha>0$ we have
   $$
      \frac{M(0,t-1)}{1+m_2(t)}
      \leq \pr{\tau < t}
      \leq \frac{M(0,(1+\alpha)t-1)}{1+m_1(\alpha t)}.
   $$
\end{lemma}
\begin{proof}%[\textbf{Proof of Lemma~\ref{lem:dettau}}]
   We apply 
   the straightforward equation
   $$
      M(0,t-1) = \sum_{i=0}^{t-1} \pr{\tau = i} \left(1+\Espace{M'(Y_i,i+1,t-1)}{Y_i}\right),
      %\label{eq:detmain}
   $$
   where the random variable~$Y_i$ denotes the relative location of~$v$ with respect 
   to~$u$ given that $\tau=i$.  Note that $Y_i \in B_r$ since the condition $\tau=i$ implies 
   that the distance between $v$ and $u$ at time $i$ is at most~$r$.  
   Using Lemma~\ref{lem:detm1m2} we obtain
   \begin{equation}
      M(0,t-1) 
      \leq \sum_{i=0}^{t-1} \pr{\tau=i} \left(1+\Espace{M'(Y_i,i+1,i+t)}{Y_i}\right) 
      \leq \pr{\tau < t} \left(1+m_2(t)\right).
      \label{eq:detbound1}
   \end{equation}
   Also, since $M(0,t-1)$ is non-decreasing with $t$, 
   we can take an arbitrary constant $\alpha>0$ and use the fact that 
   $(1+\alpha)t-1 \geq i+\alpha t$ for all $i\in [0,t-1]$ together with Lemma~\ref{lem:detm1m2}
   to write
   \begin{eqnarray}
      M(0,(1+\alpha)t-1) 
      &\geq& \sum_{i=0}^{t-1} \pr{\tau=i} \left(1+\Espace{M'(Y_i,i+1,(1+\alpha)t-1)}{Y_i}\right) \nonumber\\
      &\geq& \sum_{i=0}^{t-1} \pr{\tau=i} \left(1+\Espace{M'(Y_i,i+1,i+\alpha t)}{Y_i}\right)  \nonumber\\
      &\geq& \pr{\tau < t} \left(1+m_1(\alpha t)\right).
      \label{eq:detbound2}
   \end{eqnarray}
   The proof is completed by reorganizing (\ref{eq:detbound1}) and (\ref{eq:detbound2}).
\end{proof}

We are now in position to conclude the proof of Theorem~\ref{thm:detection}.
Plugging Lemmas~\ref{lem:detm} and~\ref{lem:detm1m2} into Lemma~\ref{lem:dettau}, and using 
Lemma~\ref{lem:det1step}, we obtain
a constant $t_0=t_0(d,s)$, and 
constants $c_1,c_2,c_3,c_4$ depending only on~$d$, such that for all $t\geq t_0$ and 
sufficiently large~$L$,
%below removed after lemmas were added to save space. They have reference to equations in the appendix.
\ignore{
   We are now in position to conclude the proof of Theorem~\ref{thm:detection}.
   Plugging Lemmas~\ref{lem:detm} and~\ref{lem:detm1m2} into~(\ref{eq:detmainineq}), we obtain
   from~(\ref{eq:deteq}) and~(\ref{eq:detlb}) that there exists a constant $t_0=t_0(d,s)$, and 
   constants $c_1,c_2,c_3,c_4$ depending only on~$d$, such that for all $t\geq t_0$ and 
   sufficiently large~$L$
}
\begin{equation}
   \exp\left(-\frac{c_1 \lambda s^2 t}{s^2 + c_2 \log t}\right)
   \leq \pr{\tdet(Q_L) \geq t}
   \leq \exp\left(-\frac{c_3 \lambda s^2 t}{s^2 + c_4 \log t}\right)
   \label{eq:detfinal1}
\end{equation}
for $d=2$, and 
\begin{equation}
   \exp\left(-c_1 \lambda s^d t\right)
   \leq \pr{\tdet(Q_L) \geq t} 
   \leq \exp\left(-c_3 \lambda s^d t\right)
   \label{eq:detfinal2}
\end{equation}
for $d\geq 3$. Theorem~\ref{thm:detection} then follows by taking the limit as $L\to\infty$.

\par\bigskip\noindent
{\bf Remark:}
As should be clear from the proof,
the upper bounds in (\ref{eq:detfinal1}) and (\ref{eq:detfinal2}) hold for arbitrary locations 
of~$u$ as long as $u$ moves independently of the locations of the nodes of~$G$. 
The lower bounds also hold in more generality: e.g.,
if $u$ moves according to Brownian motion with variance ${s'}^2\neq s^2$, 
or indeed with any motion that has sufficiently large ``variance" in all directions.
%(in which case the 
%constants in (\ref{eq:detfinal1}) and (\ref{eq:detfinal2}) depend also on the ratio~$s'/s$).
(Specifically, the lower bounds hold if the density $f'_i(\cdot,\cdot)$ for the motion of~$u$
after $i$ steps satisfies the following property: there exist positive constants $a_1=a_1(d)$ 
and $a_2=a_2(d)$ such that $f_i'(x,x+z) \geq a_1/i^{d/2}$ for all 
$z\in B_{a_2 \sqrt{i}}$, $x\in\mathbb{R}^d$, and $i \in \mathbb{Z}^+$.)
On the other hand, adding a random drift to the nodes in $G$ can change the detection time 
substantially: if each node $v$ of $G$ moves according to Brownian motion with 
drift $\mu_v$ and variance $s^2$, where the $\mu_v$ are i.i.d.\ random variables, 
then our proof can be adapted to show that under mild conditions\footnote{Note that this statement
cannot hold in full generality; if all nodes of $G$ have the same drift, then this is equivalent to the case without drift up to translations of $\mathbb{R}^d$.}
on the distribution of~$\mu_v$,
$\pr{\tdet \geq t}= \exp\left(-\Theta(t)\right)$ in all dimensions $d\geq 2$ for arbitrary locations of $u$. 
We omit the details.
\section{Percolation time}
\label{sec:percolation}
In this section we prove Theorem~\ref{thm:percolation}.
We consider a MGG~$G$ with density $\lambda>\lambda_c$ (i.e., above the
percolation threshold), and study the random variable $\tperc$ defined as the first 
time at which a node~$u$ initially placed at the origin independently of the nodes of~$G$
belongs to the infinite component of $G$.
We derive an upper bound for the tail $\pr{\tperc \geq t}$ as $t\to\infty$.

We begin by stating a proposition that will be a key ingredient in our analysis. 
We consider a large cube $Q_K \subset \mathbb{R}^d$ and tessellate it into small
cubes called ``cells."
The proposition says that, if all cells have sufficiently many nodes at a given time $i$, 
then at time $i+\Delta$ for suitably large~$\Delta$ the point process induced by
the location of the nodes contains a fresh Poisson point process with only slightly reduced
intensity inside a smaller cube~$Q_{K'}$.
We believe this result is of independent interest.  With this in mind, we state
the proposition below for a slightly more general setting than is needed here.
Its proof is deferred to the end of the section.

\begin{proposition}\label{pro:coupling}
   Fix $K > \ell >0$ and consider the cube $Q_K$ tessellated into cells of side-length~$\ell$.
   Let $\Pi_0$ be an arbitrary point process at time $0$ 
   that contains at least $\beta \ell^d$ nodes in each cell of the tessellation for some $\beta>0$.
   Let $\Pi_\Delta$ be the point process obtained at time $\Delta$ from $\Pi_0$ by allowing 
   the nodes to move according to Brownian motion with variance~$s^2$.
   Fix $\epsilon\in(0,1)$ and let $\Xi$ be a fresh Poisson point process with
   intensity $(1-\epsilon)\beta$. Then there exists a coupling of~$\Xi$ and~$\Pi_\Delta$ and
   constants $c_1,c_2,c_3$ depending only on $d$ such that, if 
   $\Delta \geq \frac{c_1 \ell^2}{s^2 \epsilon^2}$ and $K' \le K - c_2 s \sqrt{\Delta \log \epsilon^{-1}}>0$,
   the nodes of $\Xi$ are a subset of the nodes of $\Pi_\Delta$ inside the cube $Q_{K'}$ with probability
   $
      1-\frac{K^d}{\ell^d}\exp(-c_3\epsilon^2\beta\ell^d).
   $
\end{proposition}

Now we proceed to the proof of Theorem~\ref{thm:percolation}.
We first take a sufficiently small parameter $\xi>0$ 
such that $(1-\xi)^2 \lambda > \lambda_\critperc$.  (This is always possible as we
are assuming $\lambda > \lambda_\critperc$.)  
In what follows, we omit the dependencies of other
parameters on $\lambda$ and $\xi$ as we are considering them to be fixed. 

Let $H_i$ be the event that $u$ does \emph{not} belong to the infinite component at time $i$. 
Then, the event $\{\tperc \geq t \}$ is equivalent to $\bigcap_{i=0}^{t-1} H_i$. We define an integer
parameter $\Delta \geq 1$ and consider the process obtained by skipping every $\Delta$ 
time steps. (To simplify the notation we assume w.l.o.g.\ that $t/\Delta$ is an integer.) 
In other words,
instead of looking at the event $\bigcap_{i=0}^{t-1} H_i$ we consider the event 
$\bigcap_{i=0}^{t/\Delta-1} H_{\Delta i}$, which we henceforth denote by $\mathcal{H}_t$.
Since the occurrence of the event $\{\tperc \geq t\}$ implies 
$\mathcal{H}_t$ we have $\pr{\tperc \geq t} \leq \pr{\mathcal{H}_t}$.
Our goal in introducing~$\Delta$ is to allow nodes to move further between consecutive time steps;
we will choose the value of~$\Delta$ later.  

Let $C=C(d) \geq 1$ be a sufficiently large constant and
fix $L=C t(1+s)$. We will confine our attention to the cube~$Q_{2L}$. 
We take a parameter $\ell>0$ and 
tessellate $Q_{2L}$ into cubes of side-length~$\ell$ (see Figure~\ref{fig:tess}(a)). 
We refer to each such cube as a ``cell."
Later we will tie together the values of $\ell$ and $\Delta$, and will choose~$\ell$
to optimize our upper bound for $\pr{\mathcal{H}_t}$. 
For the moment we only assume that the tessellation is non-trivial in the 
sense that both~$\ell$ and $L/\ell$ are $\omega(1)$ as functions of~$t$.

For each time step $i$, the expected number of nodes inside a given cell is $\lambda \ell^d$.
We say that a cell is \textit{dense} at time $i$ if it contains 
at least $(1-\xi)\lambda \ell^d$ nodes, where $\xi$ is as defined earlier.
Let $D_i$ be the event that \emph{all} cells are dense at time $i$, and let 
$\mathcal{D}_{t} = \bigcap_{i=0}^{t/\Delta-1} D_i$. 
The lemma below shows that 
$\mathcal{D}_{t}$ occurs with high probability. 
\begin{figure}[ht]
   \begin{center}
   \input{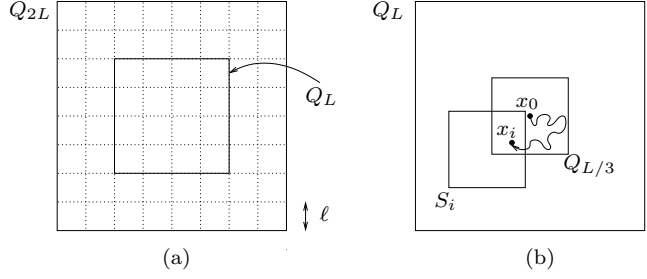}
   \caption{(a)~The cubes $Q_{2L}$ and $Q_L$ and the tessellation of $Q_{2L}$ into cubes of 
   side-length $\ell$. (b)~The cube $Q_{L/3}$, the locations $x_0$ and $x_i$ of node $u$ 
   at time steps $0$ and $i$ respectively, and the cube $S_i$.}
   \label{fig:tess}
   \end{center}
   %\vspace{-.4in}
\end{figure}

\begin{lemma}\label{lem:density}
   With the above notation,
   $
      \pr{\mathcal{D}_{t}} 
      \geq 1 - \frac{tL^d}{\Delta \ell^d}\exp\left(-\xi^2 \lambda \ell^d/2\right).
   $
\end{lemma}
\begin{proof}%[\textbf{Proof of Lemma~\ref{lem:density}}]
   At any given step $i$, by a standard large deviation bound for a Poisson 
   r.v.\ (cf.\ Lemma~\ref{lem:cbpoisson}),
   a cell has more than $(1-\xi)\lambda \ell^d$ nodes with probability 
   at least $1-\exp(-\xi^2\lambda\ell^d/2)$. The proof is completed by taking the 
   union bound over all $(L/\ell)^d$ cells and $t/\Delta$ time steps.
\end{proof}

Recall that $x_i$ is the location of $u$ at time $i$.
Define $E_i$ as the event that $x_i$ is located inside $Q_{L/3}$ at time $i$, and let 
$\mathcal{E}_t = \bigcap_{i=0}^{t/\Delta-1} E_{\Delta i}$. The next lemma bounds the probability 
that $u$ never leaves $Q_{L/3}$. 
%Note that if this event holds, then $S_i \subset Q_L$ for all $i$.
\begin{lemma}\label{lem:escape}
   There exists a constant $c=c(d)$ such that 
   $
      \pr{\mathcal{E}_t} 
      \geq 1 - \exp\left(-c t\right).
   $
\end{lemma}
\begin{proof}%[\textbf{Proof of Lemma~\ref{lem:escape}}]
   We fix a time step $i$ and then apply the union bound over time. The event $E_i$ corresponds to $u$ 
   not moving a distance more than $L/6$ in any dimension. Therefore,
   \begin{eqnarray*}
      \pr{E_i} 
      &=& \int_{Q_{L/3}} f_{i}(0,x_i)dx_i\\
      &=& \left[1 - 2\int_{L/6}^\infty \frac{1}{\sqrt{2\pi s^2 i}} \exp\left(-\frac{y^2}{2 s^2 i}\right)dy\right]^d
      \geq 1 - 2d\int_{L/6}^\infty \frac{1}{\sqrt{2\pi s^2 i}} \exp\left(-\frac{y^2}{2 s^2 i}\right)dy.
   \end{eqnarray*}
   Then we use a standard large deviation bound for the Normal distribution (see Lemma~\ref{lem:cbnormal}) to conclude that
   $$
      \pr{E_i} 
      \geq 1 - \frac{12ds\sqrt{i}}{\sqrt{2\pi} L} \exp\left(-\frac{L^2}{72 s^2i}\right).
   $$
   Since the bound above decreases with $i$ we can conclude that 
   $$
      \pr{\mathcal{E}_t} \geq 1 - \frac{t}{\Delta}\frac{12d s\sqrt{t}}{\sqrt{2\pi} L} \exp\left(-\frac{L^2}{72 s^2t}\right),
   $$
   and the result follows from $\Delta \geq 1$ and $L \geq st$.
\end{proof}

For each time step~$i$, we define~$S_i$ to be the cube~$Q_{L/3}$
shifted randomly so that $x_i$ (the location of~$u$ at time~$i$)
is uniformly random in~$S_i$ (see Figure~\ref{fig:tess}(b)).
%For each time step~$i$, we define the cube $S_i$ by translating $Q_{L/3}$ 
%so that $u$ is located u.a.r.\ inside $S_i$ at time $i$ 
%(I.e., we choose a point from $Q_{L/3}$ uniformly at random and translate
%$Q_{L/3}$ so that the chosen point is at location $x_i$). 
A {\it crossing component\/} of $S_i$ is a connected set of nodes within~$S_i$
that contains a path connecting every pair of opposite faces of~$S_i$.  (A path
connects two faces of $S_i$ if each face is within distance~$r$ 
of one of its endpoints.)

For each~$i$,
let $K_i$ be the event that all the crossing components of $S_i$ are contained in
the infinite component at time~$i$. 
(For definiteness we assume that $K_i$ holds if $S_i$ has no crossing component.)
Let $\mathcal{K}_t=\bigcap_{i=0}^{t/\Delta-1} K_{\Delta i}$.
The next lemma follows by a result of Penrose and Pisztora~\cite[Theorem~1]{PenPis}.
\begin{lemma}\label{lem:pisztora}
   For any $\lambda>\lambda_\critperc$, there exists a constant $c=c(d)$ such that
   $
      \pr{\mathcal{K}_t} 
      \geq 1 - \exp(-ct^{d-1}).
   $
\end{lemma}
\begin{proof}%[\textbf{Proof of Lemma~\ref{lem:pisztora}}]
   By stationarity we know that $\pr{K_i}$ is the same for all $i$. For any fixed~$i$, 
   \cite[Theorem~1]{PenPis} gives that
   $\pr{K_i} \geq 1 - \exp(-c'L^{d-1})$ for some constant $c'$. 
   (In fact, \cite[Theorem~1]{PenPis} handles an event more restrictive than $K_i$, which among other things considers
   \emph{unique} crossing components of $S_i$.)
   Using the union bound we obtain
   $\pr{\mathcal{K}_t} \geq 1- (t/\Delta)\exp(-c'L^{d-1})$ and the result follows since $L \geq t$.   
\end{proof}

We now proceed to derive a bound on $\pr{\mathcal{H}_{t}}$.
We take $H'_i$ to be the event that $u$ does not belong to a crossing component of $S_i$ at time $i$, and define 
$\mathcal{H}'_t = \bigcap_{i=0}^{t/\Delta-1} H'_{\Delta i}$. Note that $H'_i$ is a {\it decreasing\/} 
event, in the sense that if  $H'_i$ occurs then it also occurs after removing any arbitrary collection 
of nodes from the MGG~$G$.
Clearly $\mathcal{H}_t \cap \mathcal{K}_t \subseteq \mathcal{H}'_t \cap \mathcal{K}_t$. 
By elementary probability,
\begin{equation}
   \pr{\mathcal{H}_t} 
   \leq \prcond{\mathcal{H}'_t \cap \mathcal{D}_t}{\mathcal{E}_t} + \pr{\mathcal{D}_t^\compl} 
   + \pr{\mathcal{E}_t^\compl} + \pr{\mathcal{K}_t^\compl}.
   \label{eq:percinparts}
\end{equation}
Note that we use $\mathcal{K}_t$ only to replace $\mathcal{H}_t$ by
$\mathcal{H}'_t$ in~(\ref{eq:percinparts});
this helps to control the dependencies among time steps, since $\mathcal{H}'_t$ 
is an event restricted to the cubes $S_i$
while $\mathcal{H}_t$ is an event over the whole of~$\mathbb{R}^d$.
We use $\mathcal{E}_t$ only to ensure that $S_i\subset Q_L$, which allows
us to focus on the portion of~$G$ inside~$Q_L$.  Note that $\mathcal{E}_t$
is independent of~$G$, so this conditioning does not affect~$G$.
%Now we explain the role of $\mathcal{E}_t$.
%The only information we are going to use from the conditioning on $\mathcal{E}_t$ is that $S_i \subset Q_L$, which does not
%depend on the nodes of $G$. Therefore, $G$ is independent of $\mathcal{E}_t$.
%Since $S_i \subset Q_L$, the event $\mathcal{E}_t$ allows us to focus on that portion of~$G$
%that is inside $Q_L$ only. 

Now we set $\Delta = \lceil C^2 \ell^2/s^2\rceil$, where $C$ is the constant in the definition of $L$. 
The main step in our proof is the lemma below.
\begin{lemma}\label{lem:percolationlemma}
   Let $\lambda > \lambda_\critperc$ and $\xi>0$ be such that $(1-\xi)^2\lambda >\lambda_\critperc$.
   Let $C$ be large enough in the definition of $L$ and $\Delta$.
   There exist constants $c=c(d)$ and $t_0=t_0(d)$ such that, for all $t\geq t_0$, we have
   $$
      \prcond{ {\mathcal{H}}'_{t} \cap  \mathcal{D}_t}{\mathcal{E}_t} 
      \leq \exp\left(-c t/\Delta\right).
   $$
\end{lemma}
\begin{proof}
   We start by writing
   \begin{equation}
      \prcond{ \mathcal{H}'_{t} \cap  \mathcal{D}_t}{\mathcal{E}_t}
      \leq \prod_{i=0}^{t/\Delta-1} \mathbf{Pr}[H'_{\Delta i}\mid \mathcal{H}'_{\Delta(i-1)} \cap \mathcal{D}_{\Delta(i-1)} \cap \mathcal{E}_t].
      \label{eq:product}
   %   \leq \prod_{i=0}^{t/\Delta-1} \prcond{H'_{\Delta i}}{\mathcal{H}'_{\Delta(i-1)} \cap \mathcal{D}_{\Delta(i-1)} \cap \mathcal{E}_t}
   \end{equation}
   (Here, for notational convenience, we assume that $\mathcal{H}'_{-\Delta} \cap \mathcal{D}_{-\Delta}\cap\mathcal{E}_t=\mathcal{E}_t$.)
   
   We now derive an upper bound for 
   $\mathbf{Pr}[H'_{\Delta i} \mid \mathcal{H}'_{\Delta(i-1)} \cap \mathcal{D}_{\Delta(i-1)} \cap \mathcal{E}_t]$.
   We start with a high level overview of the proof.
   Let $\Xi_{\Delta (i-1)}$ be the (not necessarily Poisson) point process obtained from the nodes of 
   $\Pi_{\Delta (i-1)}$ (the MGG at time $\Delta(i-1)$) under
   the condition $\mathcal{H}'_{\Delta(i-1)} \cap \mathcal{D}_{\Delta(i-1)}$.
   Note that $\Xi_{\Delta(i-1)}$ is conditioned only on events that occur between time $0$ and 
   time $\Delta (i-1)$;
   therefore, the motion of the nodes of $\Xi_{\Delta (i-1)}$ from time 
   $\Delta (i-1)$ to $\Delta i$ is independent of the condition. 
   Since all cells are assumed dense at time $\Delta(i-1)$, using Proposition~\ref{pro:coupling}
   we can construct an {\it independent\/} Poisson point process 
   $\Xi'_{\Delta (i-1)}$ and couple it with 
   $\Xi_{\Delta (i-1)}$ so that at time $\Delta i$ the nodes of $\Xi'_{\Delta i}$ in $Q_L$
   are a subset of the nodes of $\Xi_{\Delta i}$.
   Moreover, we can ensure that $\Xi'_{\Delta i}$ has intensity larger than $\lambda_\critperc$ in $Q_L$, and thus conclude that 
   $u$ will belong to a crossing component of~$S_{\Delta i}$ with constant probability. 
   Using this, we can upper bound each term of the product
   in~(\ref{eq:product}) by a constant strictly smaller than~$1$, which gives
   $\prcond{ \mathcal{H}'_{t} \cap  \mathcal{D}_t}{\mathcal{E}_t} \leq \exp(-c t/\Delta)$ for some constant $c=c(d)$.
   
   Turning now to the details, we can invoke Proposition~\ref{pro:coupling} with 
   $\beta=(1-\xi)\lambda$, $\epsilon=\xi$, $K=2L$, and $K'=L$ to
   obtain that, conditioned on $\mathcal{H}'_{\Delta(i-1)} \cap \mathcal{D}_{\Delta(i-1)}$, 
   at time $\Delta i$ the nodes of the MGG in $Q_L$ contain a fresh Poisson point process
   with intensity $(1-\xi)^2\lambda>\lambda_\critperc$ with probability 
   at least $1-\exp\left(-c'\xi^2 \lambda \ell^d\right)$, for some constant $c'$. 
   Since $H'_{\Delta i}$ is a decreasing event\footnote{Note that we defined
   $H'_{\Delta i}$ in terms of crossing components precisely in order to make it
   a decreasing event; otherwise, we could have defined it in terms of the largest 
   connected component of~$S_{\Delta i}$.}, 
   if we define $H''_{\Delta i}$ as the event $H'_{\Delta i}$ restricted to the 
   nodes of the fresh Poisson point process, then $H'_{\Delta i} \subseteq H''_{\Delta i}$ and 
   $$
      \prcond{ \mathcal{H}'_{\Delta i}}{\mathcal{H}'_{\Delta(i-1)} \cap \mathcal{D}_{\Delta(i-1)} \cap \mathcal{E}_t}
      \leq \pr{ \mathcal{H}''_{\Delta i}} + \exp\left(-c'\xi^2 \lambda \ell^d\right),
   $$
   as $\mathcal{H}''_{\Delta i}$ does not depend on the condition. 
   Since the intensity of the fresh Poisson point 
   process is $(1-\xi)^2\lambda > \lambda_\critperc$, \cite[Theorem~1]{PenPis} implies that,
   with probability $1-\exp(-c'' L^{d-1})$ for some constant~$c''$, 
   a constant fraction of the volume of $S_{\Delta i}$ is within distance $r$ of at least one node 
   in a crossing component of~$S_{\Delta i}$.
   Since at time~$\Delta i$, $u$ is located uniformly at random inside $S_{\Delta i}$, 
   $u$ belongs to a crossing component of $S_{\Delta i}$ with probability at least
   $c'''>0$, where $c'''$ is a 
   constant. Hence we have
   $$
      \prcond{ \mathcal{H}'_{t} \cap  \mathcal{D}_t}{\mathcal{E}_t}
      \leq \prod_{i=0}^{t/\Delta-1} \left[1-c'''+\exp(-c'' L^{d-1})
         + \exp\left(-c'\xi^2 \lambda \ell^d\right)\right].
   $$
   Since $L$ and $\ell$ go to infinity with~$t$, for $t$ sufficiently large each factor
   in the above product can be made strictly 
   smaller than~$1$, which concludes the proof of Lemma~\ref{lem:percolationlemma}. 
\end{proof}

Finally, we plug Lemmas~\ref{lem:density}--\ref{lem:percolationlemma} into 
(\ref{eq:percinparts}) and obtain the following upper bound on $\pr{\mathcal{H}_t}$:  $$
   \pr{\tperc \geq t} \le \pr{\mathcal{H}_t} \leq \exp\left(-c t/\Delta\right) + 
               \exp\left(-c \ell^d\right) +
                              \exp\left(-c t\right) +
                              \exp(-ct^{d-1}).  $$
(Here $c=c(d)$ is a generic constant.)  In order to minimize this
upper bound we choose~$\ell$ so that $\ell^d=\Theta(t/\Delta)$, 
which yields
$$
   \pr{\tperc \geq t}  \leq \exp\left(-c t^{\frac{d}{d+2}}\right)
$$
for all sufficiently large~$t$, where $c$ is a constant depending on~$\lambda$,~$s$, and~$d$.
This completes the proof of Theorem~\ref{thm:percolation}.  
It remains to go back
and prove Proposition~\ref{pro:coupling}.

%############################################################################################
%############################################################################################
%############################################################################################
\subsubsection*{Proof of Proposition~\ref{pro:coupling}}
We will construct $\Xi$ via three Poisson point processes.
We start by defining $\Xi'_0$ as a Poisson point process over $Q_{K}$ 
with intensity $(1-\epsilon/2)\beta$.
Recall that $\Pi_0$ has at least $\beta \ell^d$ nodes in each cell of $Q_K$.
Then, in any fixed cell,
$\Xi'_0$ has fewer nodes than $\Pi_0$ if 
$\Xi'_0$ has less than $\beta\ell^d$ nodes in that cell, which by a
standard Chernoff bound (cf.\ Lemma~\ref{lem:cbpoisson}) occurs 
with probability larger than 
$1-\exp\left(-\frac{{\epsilon'}^2(1-\epsilon/2)\beta\ell^d}{2}(1-{\epsilon'}/3)\right)$
for $\epsilon'$ such that $(1+\epsilon')(1-\epsilon/2)=1$. Since $\epsilon \in (0,1)$ we 
have $\epsilon'\in(\epsilon/2,1)$, and the probability above can be bounded below by
$1-\exp\left(-c\epsilon^2\beta\ell^d\right)$ for some constant $c=c(d)$.
Let $\{\Xi'_0 \preceq \Pi_0\}$ be the event that $\Xi'_0$ has fewer nodes
than $\Pi_0$ in every cell of $Q_{K}$. 
Using the union bound over cells we obtain
\begin{equation}
   \mathbf{Pr}[\Xi'_0 \preceq \Pi_0] \geq 1- \frac{K^d}{\ell^d}\exp(-c\epsilon^2\beta\ell^d).
   \label{eq:cbcoupling}
\end{equation}

If $\{\Xi'_0 \preceq \Pi_0\}$ holds, then we can map each node of 
$\Xi'_0$ to a unique node of $\Pi_0$ in the same cell. We will now show that we can couple
the motion of the nodes in $\Xi'_0$ with the motion of their respective pairs in 
$\Pi_0$ so that the probability that an arbitrary pair is at the same location at time $\Delta$
is sufficiently large. 

To describe the coupling, let
$v'$ be a node from $\Xi'_0$ located at $y' \in Q_{K}$, and let $v$ be the pair of 
$v'$ in $\Pi_0$. Let $y$ be the location of $v$ in $Q_{K}$, and note that since $v$ and 
$v'$ belong to the same cell we have $\|y-y'\|_2 \leq \sqrt{d}\ell$. 
We will construct a function $g(z)$ that is smaller than the densities
for the motions of $v$ and $v'$ to the location $y'+z$, uniformly for $z\in \mathbb{R}^d$. 
That is, 
\begin{equation}
   g(z) \leq \min\{f_{\Delta}(y',y'+z),f_{\Delta}(y,y'+z)\}
   = \frac{1}{(2\pi s^2\Delta)^{d/2}} \exp\left(-\frac{\max\{\|z\|_2^2,\|y'+z-y\|_2^2\}}{2s^2\Delta}\right)
   \label{eq:constraintg}
\end{equation}
for all $z\in \mathbb{R}^d$.  

We set 
\begin{equation}
   g(z) = \frac{1}{(2\pi s^2\Delta)^{d/2}} \exp\left(-\frac{(\|z\|_2+\sqrt{d}\ell)^2}{2s^2\Delta}\right).
   \label{eq:g}
\end{equation}
Note that this definition satisfies~(\ref{eq:constraintg}) since by the triangle inequality 
$\|y'+z-y\|_2 \leq \|y'-y\|_2+\|z\|_2$ and $\|y'-y\|_2 \leq \sqrt{d}\ell$. 
Define $\psi= 1-\int_{\mathbb{R}^d} g(z)dz$. 
Then, with probability $1-\psi$ we 
can use the density function $\frac{g(z)}{1-\psi}$ to sample a single location 
for the position of both $v$ and $v'$ at time $\Delta$,
and then set $\Xi''_0$ to be the Poisson point process with 
intensity $(1-\psi)(1-\epsilon/2)\beta$
obtained by {\it thinning} $\Xi'_0$ (i.e., deleting each node
of $\Xi'_0$ with probability $\psi$). 
At this step we have crucially used the fact that the function
$g(z)$ in (\ref{eq:g}) is oblivious of the location of $v$ and, consequently, is independent of the 
point process $\Pi_0$. (If one were to use the maximal coupling suggested
by~(\ref{eq:constraintg}), then the thinning probability would depend on $\Pi_0$, and 
$\Xi''_0$ would not be a Poisson point process.)

Let $\Xi''_\Delta$ be obtained from $\Xi''_0$
after the nodes have moved 
according to the density function $\frac{g(z)}{1-\psi}$. 
Thus we are assured that the nodes of the Poisson point process $\Xi''_\Delta$ 
are a subset of the nodes of $\Pi_\Delta$ and are independent of the nodes of $\Pi_0$, where $\Pi_\Delta$ is 
obtained by letting the nodes of $\Pi_0$ move from time $0$ to time $\Delta$.

The next lemma shows that if $\Delta$ and $K-K'$ are large enough, 
then the integral of $g(z)$ inside the ball $B_{(K-K')/2}$ is larger than
$1-\epsilon/2$. (We are interested in the ball $B_{(K-K')/2}$ since for all $z\in Q_{K'}$ we have $z+B_{(K-K')/2} \subset Q_K$.)
\begin{lemma}\label{lem:psi}
   If $\Delta \geq c \frac{\ell^2}{s^2\epsilon^2}$ and 
   $K-K'\geq c' s \sqrt{\Delta \log \epsilon^{-1}}$ for large enough $c,c'$, we may ensure that
   $\int_{B_{(K-K')/2}} g(z) dz \geq 1-\epsilon/2$.
\end{lemma}
\begin{proof}%[\textbf{Proof of Lemma~\ref{lem:psi}}]
   Since $g(z)$ depends on $z$ only via $\|z\|_2$, we integrate over 
   $\rho=\|z\|_2$ and 
   let $a \rho^{d-1}$ be the surface area of~$B_\rho$, which gives
   $$
      \int_{B_{(K-K')/2}} g(z) dz
      = \int_{0}^{(K-K')/2} \frac{a \rho^{d-1}}{(2\pi s^2\Delta)^{d/2}} \exp\left(-\frac{(\rho+\sqrt{d}\ell)^2}{2s^2\Delta}\right) d\rho.
   $$
   Now we change variables to $\rho'=\frac{\rho+\sqrt{d}\ell}{s\sqrt{\Delta}}$ and set
   $\delta=\frac{\sqrt{d}\ell}{s\sqrt{\Delta}}=\epsilon\sqrt{d/c}$ and $K''=\frac{K-K'}{2s\sqrt{\Delta}}+\delta$
   to obtain
   \begin{equation}
      \int_{B_{(K-K')/2}} g(z) dz
      = \frac{a}{(2\pi)^{d/2}}\int_{\delta}^{K''} \left(\rho'-\delta\right)^{d-1} \exp(-{\rho'}^2/2) d\rho'.
      \label{eq:psifortaylor}
   \end{equation}
   Let $h(\delta) = \int_{\delta}^\infty (\rho'-\delta)^{d-1} \exp(-{\rho'}^2/2)d\rho'$.
   We will apply the Taylor expansion of $h(\delta)$ around $\delta=0$.
   %Note that 
   %$$
   %   \frac{a}{(2\pi)^{d/2}} h(0)
   %   = \int_{0}^\infty \frac{a{\rho'}^{d-1}}{(2\pi)^{d/2}} \exp(-{\rho'}^2/2)d\rho'
   %   = \int_{\mathbb{R}^d} \frac{1}{(2\pi)^{d/2}} \exp(-\|z\|_2^2/2)dz
   %   =1.
   %$$
   %The derivative of $h(\delta)$ is
   Note that $\frac{a}{(2\pi)^{d/2}} h(0) = 1$, and 
   the derivative of $h(\delta)$ is
   $$
      h'(\delta)
      = -(d-1)\int_{\delta}^{\infty} \left(\rho'-\delta\right)^{d-2} \exp(-{\rho'}^2/2) d\rho'.
   $$
   In particular, the derivative increases with $\delta$, and by Taylor's theorem, 
   $h(\delta) \geq h(0) + \delta h'(0)$. Therefore, we have
   $$
      \frac{a}{(2\pi)^{d/2}} h(\delta)
      \geq \frac{a}{(2\pi)^{d/2}} \left(h(0) + \delta h'(0)\right)
      = 1-\delta \frac{ah'(0)}{(2\pi)^{d/2}}.
   $$
   Note that $\delta\frac{ah'(0)}{(2\pi)^{d/2}}$ depends on the dimension only and can, 
   for example, be made 
   smaller than $\epsilon/4$ for sufficiently large $c$. Using equation~(\ref{eq:psifortaylor}) we get
      \begin{equation}
      \int_{B_{(K-K')/2}} g(z) dz
      \geq 1-\epsilon/4 - \frac{a}{(2\pi)^{d/2}}\int_{K''}^{\infty} \left(\rho'-\delta\right)^{d-1} \exp(-{\rho'}^2/2) d\rho'.
      \label{eq:almostthere}
   \end{equation}
   Now note that $\rho'^{d-1}\exp(-{\rho'}^2/2) \le c''\exp(-{\rho'}^2/3)$ uniformly for $\rho'\in[0,\infty)$,
   where $c''$ is a constant depending only on~$d$.  Thus we have
   \begin{equation}
      \int_{K''}^\infty \frac{a}{(2\pi)^{d/2}} (\rho'-\delta)^{d-1}\exp(-{\rho'}^2/2)d\rho'
            \leq  \frac{c''a}{(2\pi)^{d/2}} \int_{K''}^\infty\exp(-{\rho'}^2/3)d\rho'\nonumber\\
      \leq c''' \exp(-{K''}^2/3), \label{eq:lastg}
   \end{equation}
   for a constant $c'''$ depending only on~$d$.  
%   By the condition on $K-K'$, we have ${K''}^2\geq c'' \log \epsilon^{-1}$ for some sufficiently large $c''$.
%   Thus, we have
%   $x^{d-1}\exp(-x^2/2) \leq \exp(-x^2/3)$ for all $x \geq K''$, which gives 
%   \begin{eqnarray}
%      \int_{K''}^\infty \frac{a}{(2\pi)^{d/2}} (\rho'-\delta)^{d-1}\exp(-{\rho'}^2/2)d\rho'
%      &\leq& \int_{K''}^\infty \frac{a}{(2\pi)^{d/2}} {\rho'}^{d-1}\exp(-{\rho'}^2/2)d\rho'\nonumber\\
%      &\leq& \int_{K''}^\infty \frac{a}{(2\pi)^{d/2}} \exp(-{\rho'}^2/3)d\rho'\nonumber\\
%      &\leq& c''' \exp(-{K''}^2/3), \label{eq:lastg}
%%      &\leq& \frac{3a}{2(2\pi)^{d/2} K''} \exp(-{K''}^2/3), \label{eq:lastg}
%   \end{eqnarray}
%   for a constant~$c'''$.
%   where in the last step we used the inequality ${\rho'}^2 \geq 2K''\rho''-{K''}^2$.
   For sufficiently large $c'$ (and thus $K''$), the right hand side in (\ref{eq:lastg}) can be made
   smaller than $\epsilon/4$.
   Plugging this into~(\ref{eq:almostthere}) yields the lemma.
\end{proof}

When 
$\{\Xi'_0 \preceq \Pi_0\}$ holds, 
$\Xi''_{\Delta}$ consists of a subset of the nodes of $\Pi_\Delta$.
Note that $\Xi''_\Delta$ is a \emph{non-homogeneous} Poisson point process over $Q_K$. 
It remains to show that the intensity of $\Xi''_\Delta$ is strictly larger than $(1-\epsilon)\beta$ in $Q_{K'}$ so 
that $\Xi$ can be obtained from $\Xi''_\Delta$ via thinning; since $\Xi''_\Delta$ is independent of 
$\Pi_0$, so is $\Xi$.

For $z \in \mathbb{R}^d$, let $\mu(z)$ be the intensity of $\Xi''_{\Delta}$. 
Since $\Xi''_0$ has no node outside $Q_{K}$, we obtain for any $z\in Q_{K'}$,
$$
   \mu(z) 
   \geq (1-\psi)(1-\epsilon/2)\beta \int_{z+B_{(K-K')/2}} \frac{g(z-x)}{1-\psi} dx
   = (1-\epsilon/2)\beta \int_{B_{(K-K')/2}}g(x) dx,
$$
where the inequality follows since $z+B_{(K-K')/2} \subset Q_{K}$ for all $z \in Q_{K'}$. 
From Lemma~\ref{lem:psi}, we have
$\int_{B_{(K-K')/2}} g(x) dx \geq 1-\epsilon/2$. We then obtain 
$\mu(z) \geq (1-\epsilon/2)^2\beta \geq (1-\epsilon)\beta$, which is the intensity of $\Xi$.
Therefore, when $\{\Xi'_0 \preceq \Pi_0\}$ holds, which occurs with probability given by (\ref{eq:cbcoupling}),
the nodes of $\Xi$ are a subset of the nodes of $\Pi_\Delta$, which completes the proof of Proposition~\ref{pro:coupling}.

%############################################################################################
%############################################################################################
%############################################################################################
\section{Broadcast time}
\label{sec:broadcast}
In this section we use Theorem~\ref{thm:percolation} to prove Corollary~\ref{cor:broadcast} for 
a finite mobile network of volume $n/\lambda$. 
%Due to lack of space, the proof is in the appendix.

We may relate the MGG model on the torus to a model on~$\rset^d$ as follows.  
Let $S_n$ denote the cube $Q_{(n/\lambda)^{1/d}}$.  
The initial distribution of the nodes is a Poisson point process over~$\rset^d$
with intensity~$\lambda$ on~$S_n$ and zero elsewhere.
We allow the nodes to move according to Brownian motion over~$\rset^d$
as usual, and at each time step we project the location of each node onto~$S_n$
in the obvious way.
   
Now let $t= C \log^{1+2/d}n$ for some sufficiently large constant $C=C(d)$. 
The proof proceeds in three stages.  First, we show that for any fixed $i\in[0,t-1]$,
the giant component of $G_i$ has at least one node in common with the giant 
component of~$G_{i+1}$. This means that, once the message has reached the giant
component, it will reach any node~$v$ as soon as~$v$ itself belongs to the giant
component.  Thus we can bound the broadcast time by (twice) the time until
all nodes of~$G$ have been in the giant component.

In order to prove the above claim, let $\epsilon>0$ be sufficiently small so
that $(1-\epsilon)\lambda>\lambda_\critperc$. We use 
the thinning property to split $\Pi_i$ into two Poisson point processes, 
$\Pi_i'$ and~$\Pi_i''$, with intensities $(1-\epsilon)\lambda$ and $\epsilon \lambda$ 
respectively. Let $G_i'$ and $G_{i+1}'$ be the RGGs induced by~$\Pi'_i$ and $\Pi'_{i+1}$
respectively.  Then with probability $1-e^{-\Theta(n^{1-1/d})}$ both $G_i'$ and $G_{i+1}'$ 
contain a giant component~\cite{PenPis}.   We show that 
at least one node from $\Pi_i''$ belongs to both giant components. 
For any node $v$ of $\Pi_i''$, the probability that $v$
belongs to the giant component of $G_i'$ is larger than some constant $c=c(d)$. 
Moreover, using the FKG inequality we can show that $v$ belongs to the giant 
components of both $G_i'$ and $G_{i+1}'$ with probability larger than~$c^2$. 
Therefore, using the thinning property again, we can show that the nodes from~$\Pi_i''$ 
that belong to the giant components of both $G_i'$ and $G_{i+1}'$ form a Poisson point 
process with intensity $\epsilon \lambda c^2$, since $c$ does not depend on $\Pi_i''$. 
Hence, there will be at 
least one such node inside~$S_n$ with probability $1-e^{-\epsilon c^2 n}$, and this 
stage is concluded by taking the union bound over time steps~$i$.
   
Our ultimate goal is to show that if  $\tperc \leq t$ then all nodes of $G$ receive the 
message being broadcast within $2t$ steps w.h.p.
We proceed to the second stage of the proof, and show that the tail bound on $\tperc$
from Theorem~\ref{thm:percolation} also holds when applied to the finite region~$S_n$
defined above.  Note that all the derivations in the proof of Theorem~\ref{thm:percolation} 
were restricted to the cube~$Q_{2L}$, where $L$ was defined near the beginning of
Section~\ref{sec:percolation}.  Therefore, it is enough to show that $Q_{2L}$ is contained 
inside~$S_n$ (so that the toroidal boundary conditions do not affect the result).  But this
holds for all sufficiently large~$n$ since $L=O(t)=O(\log^{1+2/d}n)$ while $S_n$ has 
side-length $(n/\lambda)^{1/d}$.
   
The last stage of the proof consists of showing that adding a node~$u$ at the origin and 
calculating its percolation time (as we did in Theorem~\ref{thm:percolation}) is equivalent 
to calculating the percolation time of an arbitrary node of $G$. 
Note that, by a Chernoff bound, $G$ has at most $(1+\delta)n$ nodes with probability 
larger than $1-e^{-\Omega(n)}$ for any fixed $\delta>0$.  These nodes are indistinguishable, 
so letting $\rho$ be the probability that an arbitrary node has percolation time at 
least~$t$, we can use the union bound to deduce that this applies to {\it at least one\/}
node in $G$ with probability at most $(1+\delta)n\rho$.  Let $v$ be an arbitrary node. 
In order to relate $\rho$ to the result of Theorem~\ref{thm:percolation}, 
we can use translation invariance and assume that $v$ is at the origin. 
Then, by the ``Palm theory" of Poisson point processes~\cite{SKM}, $\rho$~is equivalent
to the tail of the percolation time for a node added at the origin, which is precisely
$\pr{\tperc\ge t}$.  Thus finally, using Theorem~\ref{thm:percolation} we get
$\rho \leq \exp(-c t^\frac{d}{d+2})$, which can be made $o(1/n)$ 
by setting $C$ sufficiently large. This completes the proof of Corollary~\ref{cor:broadcast}.
   
\par\bigskip\noindent
{\bf Remark:}
It is easy to see that the above result also holds in the case where the MGG has 
\emph{exactly} $n$ nodes. 
The proof above shows that, by setting $C$ large enough, we can ensure 
$\pr{\tbroad \geq t}=o(1/n)$ for the given value of~$t$. 
Also, it is well known that a Poisson random variable with mean~$n$ takes the value~$n$
with probability $p=\Theta(1/\sqrt{n})$. Therefore, for a MGG with exactly $n$ nodes, 
we have $\Pr[\tbroad < t]=\frac{p-o(1/n)}{p}=1-o(1/\sqrt{n})$.

%############################################################################################
%############################################################################################
%############################################################################################
\section{Some open questions}
We hope our work will promote further mathematical research on mobile networks. 
Some natural open questions related to our results include the following:

\begin{enumerate}\addtolength{\itemsep}{-0.1in}
   \item What is the tight asymptotic behavior of the tail of the percolation time~$\tperc$?
We conjecture that the upper bound of Theorem~\ref{thm:percolation} can be tightened to match
the lower bounds provided by $\pr{\tdet \geq t}$ in Theorem~\ref{thm:detection}.

   \item Let the \emph{coverage time} $\tcov$ be the time until {\it all\/} points of a finite region~$S$
are detected by the MGG.  How does $\tcov$ behave?  

\item Let $T_{\rm comm}$ denote the time until a specific node~$u$ is able to send a message
to a target node~$v$, assuming that the other nodes cooperate maximally to achieve this following
the broadcast protocol of Section~\ref{sec:broadcast}.   Since $T_{\rm comm}$ is bounded
above by the time until both~$u$ and~$v$ simultaneously belong to the giant component, 
following along the lines of the proof of Theorem~\ref{thm:percolation} for the tail of 
$\tperc$ we can show that
%an application of the FKG inequality shows that 
%$T_{\rm comm}$ inherits the same upper bound on its tail as 
$\pr{T_{\rm comm} \geq t }=\exp(-\Omega(t^\frac{d}{d+2}))$.
%$T_{\rm perc}$.  
Can one show a substantially better upper bound on the tail
of~$T_{\rm comm}$?
%   \item Let $T_{\rm 2perc}$ denote the time until two given nodes simultaneously belong to
%the infinite component; this bounds the time until the nodes are able to communicate.
%How does $T_{\rm 2perc}$ behave?  (This is essentially a question about correlations in~$\tperc$.)
   
   \item What is the broadest (realistic) class of mobility models (i.e., beyond the Brownian
motion we consider in this paper) for which our results still apply?
   
   \item In our result on the broadcast time (Corollary~\ref{cor:broadcast}), we assumed
that messages can travel instantaneously throughout each connected component, which is
reasonable in many applications (where, e.g., transmission through the air is effectively
instantaneous in comparison to the motion of the nodes).  How is this 
result affected by the assumption that messages may only travel a limited number of hops 
at each time step?
   
\end{enumerate}
%############################################################################################
%############################################################################################
%############################################################################################
\section*{Acknowledgments}
We thank Yuval Peres for helpful input on continuum percolation,
and in particular for pointing out to us the connection between the detection problem
and the Wiener sausage.  We also thank David Tse for useful discussions on
mobile wireless networks.
%############################################################################################
%############################################################################################
%############################################################################################

%############################################################################################
%############################################################################################
%############################################################################################
\appendix

%############################################################################################
%############################################################################################
%############################################################################################
\section{Standard large deviation results}

We use the following standard Chernoff bounds and large deviation results.

\begin{lemma}[Chernoff bound for Poisson]\label{lem:cbpoisson}
   Let $P$ be a Poisson random variable with mean $\lambda$. Then, for any 
   $0<\epsilon<1$,
   $$
      \pr{P \geq (1+\epsilon) \lambda} \leq  \exp\left(-\frac{\lambda \epsilon^2}{2}(1-\epsilon/3)\right),
   $$
   and
   $$
      \pr{P \leq (1-\epsilon) \lambda} \leq  \exp\left(-\frac{\lambda \epsilon^2}{2}\right).
   $$
\end{lemma}
%\begin{proof}
%   Clearly $\E{e^{\theta P}}=\exp(-\lambda+\lambda e^\theta)$ and 
%   $\E{e^{-\theta P}}=\exp(-\lambda+\lambda e^{-\theta})$. Then
%   $$
%      \pr{P \geq (1+\epsilon)\lambda} 
%      \leq \exp(-\lambda+\lambda e^\theta-\theta(1+\epsilon)\lambda)
%      = \exp(-\lambda(1-e^{\theta}+(1+\epsilon)\theta)).
%   $$
%   We set $\theta = \log(1+\epsilon)$ to optimize the right hand side above and use the 
%   inequality $(1+x)\log(1+x) \geq x+x^2/2-x^3/6$. Similarly,
%   $$
%      \pr{P \leq (1-\epsilon)\lambda} 
%      \leq \exp(-\lambda+\lambda e^{-\theta}+\theta(1-\epsilon)\lambda)
%      = \exp(-\lambda(\epsilon+(1-\epsilon)\log(1-\epsilon))),
%   $$
%   by setting $\theta= -\log(1-\epsilon)$. The proof is completed by applying the inequality 
%   $(1-\epsilon)\log(1-\epsilon) \geq -\epsilon+\epsilon^2/2$.
%\end{proof}

%\begin{lemma}[Chernoff bound for binomial]\label{lem:cbbinomial}
%   Let $B$ be a binomial random variable with parameters $N$ and $p$. Then, 
%   for  any $0<\epsilon<1$,
%   $$
%      \pr{B \geq (1-\epsilon) pN} \geq 1- e^{-\epsilon^2 Np/2}.
%   $$
%\end{lemma}
%\begin{proof}
%   Proof in \cite[Theorem~4.5]{MU}
%\end{proof}

\begin{lemma}[Large deviation for normal]\label{lem:cbnormal}
   Let $N$ be a Normal random variable with mean $0$ and variance $\sigma^2$. Then, for any $x\geq 0$,
   $$
      \pr{N \geq x} \leq \frac{\sigma}{\sqrt{2\pi} x} \exp\left(-\frac{x^2}{2\sigma^2}\right).
   $$
\end{lemma}
%\begin{proof}
%   Clearly, $\pr{N \geq x} = \int_x^\infty \frac{1}{\sqrt{2\pi \sigma^2}} \exp\left(-\frac{y^2}{2\sigma^2}\right)dy$.
%   Using the Taylor expansion of $y^2$ around $x$ we get $y^2 \geq 2xy - x^2$ and therefore
%   $$
%      \pr{N \geq x} 
%      \leq \exp\left(\frac{x^2}{2\sigma^2}\right) \int_x^\infty \frac{1}{\sqrt{2\pi \sigma^2}} \exp\left(-\frac{2xy}{2\sigma^2}\right)dy.
%      = \frac{\sigma}{\sqrt{2\pi}x} \exp\left(-\frac{x^2}{2\sigma^2}\right).
%   $$
%\end{proof}

\end{document}